\documentclass{amsart}[12pt]
\usepackage{amssymb,amsmath}
\usepackage{enumerate}
\usepackage[english]{babel}
\usepackage{color}

\newcommand{\Om} {\Omega}
\newcommand {\ep} {\varepsilon}
\newcommand {\om} {\omega}
\newcommand {\sg} {\sigma}
\newcommand \ii {\infty}
\newcommand {\dt} {\delta}
\newcommand {\al} {\alpha}
\newcommand {\bt} {\beta}
\newcommand {\lb} {\lambda}
\newcommand {\Lb} {\Lambda}

\newcommand {\ol} {\overline}
\newcommand {\sm} {\setminus}
\newcommand {\su} {\subset}
\newcommand {\wt} {\widetilde}
\newcommand {\wh} {\widehat}

\newcommand {\cal} {\mathcal}
\newcommand {\mb} {\mathbf}

\newtheorem{teo}{Theorem}[section]
\newtheorem{pro}{Proposition}[section]
\newtheorem{cor}{Corollary}[section]

\newtheorem{lm}{Lemma}[section]

\theoremstyle{definition}
\newtheorem{rem}{Remark}[section]

\title{Individual ergodic theorems \\
for infinite measure}
\keywords{Infinite measure, Dunford-Schwartz pointwise ergodic theorem, Return Times theorem, bounded Besicovitch sequence, fully symmetric space}
\subjclass[2010]{47A35(primary), 37A30(secondary)}
\begin{document}
\date{July 9, 2019}

\begin{abstract}
Given a $\sigma$\,-\,finite infinite measure space $(\Omega,\mu)$, it is shown that any Dunford-Schwartz operator $T:\, \cal L^1(\Omega)\to\cal L^1(\Omega)$ can be uniquely extended to the space $\cal L^1(\Omega)+\cal L^\ii(\Omega)$. This allows to find the largest subspace $\cal R_\mu$ of $\cal L^1(\Omega)+\cal L^\ii(\Omega)$ such that the ergodic averages $\frac1n\sum\limits_{k=0}^{n-1}T^k(f)$ converge almost uniformly (in Egorov's sense) for every $f\in\cal R_\mu$ and every Dunford-Schwartz operator $T$. Utilizing this result, almost uniform convergence of the averages $\frac1n\sum\limits_{k=0}^{n-1}\bt_kT^k(f)$ for every $f\in\cal R_\mu$, any Dunford-Schwartz operator
$T$ and any bounded Besicovitch sequence $\{\bt_k\}$ is established. Further, given a measure preserving transformation $\tau:\Omega\to\Omega$, Assani's extension of Bourgain's Return Times theorem to $\sigma$\,-\,finite measure is employed to show that for each $f\in\cal R_\mu$ there exists a set $\Omega_f\subset\Omega$ such that $\mu(\Omega\setminus\Omega_f)=0$ and the averages $\frac1n\sum\limits_{k=0}^{n-1}\bt_kf(\tau^k\om)$ converge for all $\om\in\Omega_f$ and any bounded Besicovitch sequence $\{\bt_k\}$. Applications to fully symmetric subspaces $E\subset\cal R_\mu$ are given.
\end{abstract}

\author{VLADIMIR CHILIN, \ DO\u GAN \c C\" OMEZ, \ SEMYON LITVINOV}
\address{The National University of Uzbekistan, Tashkent, Uzbekistan}
\email{vladimirchil@gmail.com; chilin@ucd.uz}
\address{North Dakota State University, P.O.Box 6050, Fargo, ND, 58108, USA}
\email{dogan.comez@ndsu.edu}
\address{Pennsylvania State University \\ 76 University Drive \\ Hazleton, PA 18202, USA}
\email{snl2@psu.edu}

\maketitle

\section{Introduction}

The celebrated Dunford-Schwartz and Wiener-Wintner-type ergodic theorems are two of the major
themes of ergodic theory.  Due to their fundamental roles, these theorems have been revisited ever since their first appearance.  For instance, Garcia \cite{ga} gave an elegant self-contained proof of Dunford-Schwartz theorem, and Assani \cite{as0,as} extended Bourgain's Return Times theorem to $\sg$-finite setting.

In the case of infinite measure, one can ask

\noindent
$(A)$ whether Dunford-Schwartz pointwise ergodic theorem is valid for some functions within the space $\cal L^1+\cal L^\ii$ but outside the union of spaces $\cal L^p$, $1\leq p<\ii$;

\noindent
$(B)$ whether pointwise convergence in Dunford-Schwartz theorem can be replaced by generally stronger almost uniform (in Egorov's sense) convergence.

To answer $(A)$, one needs to first extend a Dunford-Schwartz operator $T:\cal L^1\to \cal L^1$ to the space $\cal L^1+\cal L^\ii$. Thus, we begin by showing, in Section 3, Theorem \ref{t31}, that such an extension $\wt T$ exists and is unique if\, $\wt T|_{\cal L^\ii}$ is $\sg(\cal L^\ii,\cal L^1)$\,-\,continuous.

This fact allows us to assume without loss of generality that any Dunford-Schwartz operator $T$ is defined on the entire space $\cal L^1+\cal L^\ii$. With this assumption, positive solutions to $(A)$ and $(B)$ can be found in \cite[Theorem 3.1]{cl1}, where it was assumed a-priory that $T$ acted in the space $\cal L^1+\cal L^\ii$. In fact, the largest subspace (denoted there by $\cal R_\mu$) of $\cal L^1+\cal L^\ii$ in which the ergodic averages converge almost uniformly was found (see \cite[Theorem 3.4]{cl1}; also, \cite{cl}, \cite{kk}).

In Section 4, we use this result to show almost uniform convergence of Besicovitch weighted ergodic averages in $\cal R_\mu$ (see Theorem \ref{t44}).

In Section 5, we utilize Assani's extension of Return Times theorem to $\sg$-finite measure to show that Wiener-Wintner ergodic theorem holds in $\cal R_\mu$ with the weights $\{\lb^k\}$, $\lb\in\mathbb C_1$, expanded to the set all bounded Besicovitch sequences $\{\bt_k\}$ (see Theorem \ref{t57}).

Section 6 of the article is devoted to applications of the above results to fully symmetric spaces $E\su \cal L^1+\cal L^\ii$ such that
$\mathbf 1\notin E$. It is demonstrated that the class of fully symmetric spaces $E$ with $\mathbf 1\notin E$ is significantly wider than the class of $L^p$\,-\,spaces, $1\leq p<\ii$, including well-known Orlicz, Lorentz and Marcinkiewicz spaces of measurable functions.

\section{Preliminaries}
Let $(\Om,\cal A,\mu) $ be a $\sg$\,-\,finite measure space and let $\cal L^0 =\cal L^0(\Om)$ be the $*$-algebra of equivalence classes of almost everywhere (a.e.) finite complex-valued measurable functions on $\Om$. Given $1\leq p\leq\ii$, let $\cal L^p\su\cal L^0$ be the $L^p$-space on $\Om$ equipped with the standard Banach norm $\|\cdot\|_p$.

A net $\{f_\al\}\su\cal L^0$ is said to converge {\it almost uniformly (a.u.)} to $f\in\cal L^0$ (in Egorov's sense) if for every
$\ep>0$ there exists a set $G\su\Om$ such that $\mu(\Om\sm G)\leq\ep$ and $\|(f-f_\al)\chi_G\|_\ii\to 0$, where $\chi_G$ is the  characteristic function of  set $G$. It is clear that every a.u. convergent net converges almost everywhere (a.e.) and that the converse is not true in general.

Define
\[
\cal R_\mu=\left\{f \in\cal L^1+\cal L^\ii: \ \mu\{|f|> \lb\}<\ii \text{ \ for all \ } \lb>0\right\}.
\]

It is clear that $\cal L^p\su\cal R_\mu$ for each $1\leq p<\ii$. On the other hand, one can verify that if, for example,
$\Om=[1,\ii)$ equipped with Lebesgue measure and $f\in\cal L^\ii(\Om)$ is given by
\[
f(\om)=\sum\limits_{k=1}^\ii 2^{-k}\om^{-1/k},
\]
then $\lim\limits_{\om\to\ii}f(\om)=0$, that is, $f\in\cal R_\mu(\Om)$, but $f\notin\cal L^p(\Om)$ for all $1\leq p<\ii$.

The following characterization of $\cal R_\mu$ is crucial.

\begin{pro}\label{p21}
Let $f\in\cal L^1+\cal L^\ii$. Then
$f\in\cal R_\mu$ if and only if for each $\ep>0$ there exist $g_\ep\in \cal L^1$  and $h_\ep\in\cal L^\ii$ such that
\[
f=g_\ep+h_\ep\text{ \ \  and \ \ }\| h_\ep\|_\ii\leq\ep.
\]
\end{pro}

\begin{proof} Pick $f\in \cal R_\mu$ and let
\[
\Om_\ep=\{ |f|>\ep\}, \ \ g_\ep=f \,\chi_{\Om_\ep}, \ \ h_\ep=f\,\chi_{\Om \sm \Om_\ep}.
\]
Then $\| h_\ep\|_\ii\leq\ep$;  besides, as $f\in\cal L^1+ \cal L^\ii$, we have
\[
f=g_\ep+h_\ep=g+h
\]
for some $g\in \cal L^1$, $h\in \cal  L^\ii$. Therefore, since $f\in \cal R_\mu$, we have $\mu(\Om_\ep)<\ii$, which implies that
\[
g_\ep=g\,\chi_{\Om_\ep}+(h-h_\ep)\,\chi_{\Om_\ep}\in\cal L^1.
\]

Conversely, let $f\in\cal L^1+\cal L^\ii$, $\lb>0$, and denote $E=\{|f|>\lb\}$. Let $g_{\lb/2}\in\cal L^1$ and
$h_{\lb/2}\in\cal L^\ii$ be such that
\[
f=g_{\lb/2}+h_{\lb/2}\text{ \ \  and \ \ }\| h_{\lb/2}\|_\ii\leq\frac{\lb}2.
\]
Then we have $|f|\chi_E\leq|g_{\lb/2}|\chi_E+|h_{\lb/2}|\chi_E$, implying that
\[
\begin{split}
\mu\{|f|\chi_E>\lb\}&\leq\mu\left\{|g_{\lb/2}|\chi_E>\frac{\lb}2\right\}+\mu\left\{|h_{\lb/2}|\chi_E>\frac{\lb}2\right\}\\
&=\mu\left\{|g_{\lb/2}|\chi_E>\frac{\lb}2\right\}<\ii.
\end{split}
\]
\end{proof}

\begin{pro}\label{p22}
$\cal R_\mu$ is closed with respect to a.u. convergence.
\end{pro}
\begin{proof}
Let $\cal R_\mu\ni f_\al\to f$ a.u. Fix $\lb>0$ and denote $F=\{|f|>\lb\}$. Let $\ep>0$. Then there is $E\su\Om$ such that
\[
\mu(\Om\sm E)<\ep\text{ \ \ and \ \ } \|(f-f_\al)\chi_E\|_\ii\to 0.
\]
Since $ \|(f-f_\al)\chi_{F\cap E}\|_\ii\to 0$ and
\[
\|(f-f_\al)\chi_{F\cap E}\|_\ii\ge|f\chi_{F\cap E}-f_{\al}\chi_{F\cap E}|\ge \big|\ |f|\chi_{F\cap E}-|f_{\al}|\chi_{F\cap E} \ \big|
\]
it follows from $|f|\chi_{F\cap E}>\lb$ that there exists $\al_0$ such that
$|f_{\al_0}|\chi_{F\cap E}>\lb$. Therefore, as $f_{\al_0}\in\cal R_\mu$, we have $\mu(F\cap E)<\ii$, implying that
$\mu(F)<\ii$.
\end{proof}

\section{Extension of a Dunford-Schwartz operator to  $\cal L^1+\cal L^\ii$}
A linear operator $T: \cal L^1\to \cal L^1$ is called a {\it Dunford-Schwartz operator} (see \cite[Ch.\,VIII, \S\,6]{ds}, \cite{ga}, \cite[Ch.\,4, \S\S\,4.1, 4.2 ]{kr}), whereas we write $T\in DS$, if
\[
\| T(f)\|_1\leq\| f\|_1\, \ \ \forall \ f\in \cal L^1\text{ \ \ and \ \ } \| T(f)\|_\ii\leq\| f\|_\ii\, \ \ \forall \ f\in \cal L^\ii\cap \cal L^1.
\]

Given $\cal L\su\cal L^0$, set $\cal L_+=\{f\in\cal L:\, f\ge 0\}$. If $T\in DS$ is such that $T(\cal L^1_+)\su\cal L^1_+$, then we say that $T$ is positive and write $T\in DS^+$.

We will need the following well-known properties of a bounded linear operator $T:\cal L^1\to \cal L^1$
($T:\cal L^\ii \to \cal L^\ii$)  (see, for example, \cite[Ch.\,4, \S\,4.1, Theorem 1.1, Proposition 1.2\,(d), Theorem 1.3]{kr}).

\begin{pro}\label{p31}
For any bounded linear operator \ $T:\cal L^1\to \cal L^1$ ($T:\cal L^\ii \to \cal L^\ii$) there exists a unique positive
bounded linear operator $|T|:\cal L^1\to \cal L^1$  (respectively, $|T|:\cal L^\ii \to \cal L^\ii$) such that
\begin{enumerate} [(i)]
\item $ \|\,|T|\,\| = \| T\|$;

\item  $|T^k(f)|\leq |T|^k(|f|)$, $k=1,2, \dots$, $\forall \ f\in \cal L^1$ (respectively, $\forall \ f\in\cal L^\ii$);

\item  $|T^*|=|T|^*$, where $T^*: \cal L^\ii \to \cal L^\ii$ is the adjoint operator of an operator $T:\cal L^1\to \cal L^1$.
\end{enumerate}
\end{pro}

The operator $|T|$ is called the {\it linear modulus} of $T$.

We will also utilize the next fact, which can be found, for example, in \cite[Corollary 2.9]{pa}.
\begin{teo}\label{t30}
Let $\cal A$ and $\cal B$ be $C^*$\,-\,algebras with unit $\mb 1$, and let $T:\cal A\to\cal B$ be a positive linear map. Then $\|T\|=\|T(\mb 1)\|$.
\end{teo}

In what follows, we denote $\mb 1=\chi_\Om$.
\begin{teo}\label{t31}
For any  Dunford-Schwartz operator\; $T: \cal L^1\to \cal L^1$ there exists a unique linear operator\;
$\wt T:\cal L^1+\cal L^\ii\to \cal L^1+ \cal L^\ii$ such that
\[
\wt T(f)=T(f) \ \ \forall \ f\in \cal L^1, \ \ \ \|\wt T(f)\|_\ii\leq \|f\|_\ii \ \ \forall \ f\in \cal L^\ii,
\]
and $\wt T|_{\cal L^\ii}$ is $\sg(\cal L^\ii,\cal L^1)$\,-\,continuous.
\end{teo}
\begin{proof}
Assume first that  $T\in DS^+$. Since $ (\cal L^1)^*=\cal L^\ii$, the adjoint operator $T^*$ acts in  $\cal L^\ii$ and is
$\sg(\cal L^\ii,\cal L^1)$\,-\,continuous. Moreover, since
\[
\int_\Om T^*(f)g\,d\mu=\int_\Om fT(g)\,d\mu\ \ \ \forall \ f\in \cal L^\ii, \ g\in \cal L^1,
\]
it follows that the linear operator $T^*$ is positive.

Choose $F_n\su\Om$, $n=1,2,\dots$, satisfying
\[
F_n\su F_{n+1}, \ \ \mu(F_n)<\ii\, \ \ \forall \ n\in\mathbb N \text{ \ \ and \ \ }\bigcup\limits_{n=1}^\ii F_n=\Om.
\]
As $0\leq T(\chi_{F_n}) \leq\mathbf 1$  for each $n$, given $f\in\cal L^1\cap\cal L^\ii_+$, it follows that
\[
\begin{split}
\|T^*(f)\|_1&=\int_\Om T^*(f)\,d\mu=\lim_{n\to\ii}\int_\Om T^*(f)\chi_{F_n}d\mu\\
&=\lim_{n\to\ii}\int_\Om fT(\chi_{F_n})\,d\mu\leq\int_\Om fd\mu=\|f\|_1.
\end{split}
\]
Therefore, \,$T^*$ is $\| \cdot \|_1$\,-\,continuous on $\cal L^1\cap\cal L^\ii_+$, hence on $\cal L^1\cap\cal L^\ii$. Since $\cal L^1\cap\cal L^\ii$ is dense in $\cal L^1$, $T^*$ uniquely extends to a positive linear $\| \cdot \|_1$\,-\,continuous operator $\wh{T^*}: \cal L^1\to \cal L^1$.

Next, replacing in the above argument $T$ by $\wh{T^*}$, we uniquely extend the operator  $(\wh{T^*})^*|_{\cal L^1\cap\cal L^\ii}: \cal L^1\cap\cal L^\ii \to \cal L^1\cap\cal L^\ii$ to a positive $\|\cdot\|_1$\,-\,continuous linear operator $\wt T:\cal L^1\to\cal L^1$. Since
\[
\int_\Om f (\wh{T^*})^*(g) d \mu = \int_\Om \wh{T^*}(f) g d \mu=\int_\Om T^*(f) g d \mu = \int_\Om f T(g) d \mu \ \ \ \forall\, \ f, g \in \cal L^1\cap\cal L^\ii,
\]
it follows that $\wt T(f)=(\wh{T^*})^*(f)  = T(f)$ for all $f\in \cal L^1\cap\cal L^\ii$. Consequently, $\wt T$ coincides with $T$ on $\cal L^1$.

Furthermore, as $\wt T|_{\cal L^\ii \cap\cal L^1} = (\wh{T^*})^*|_{\cal L^\ii \cap \cal L^1}$ is
$\sg(\cal L^\ii,\cal L^1)$\,-\,continuous and $\cal L^1\cap\cal L^\ii$ is $\sg(\cal L^\ii,\cal L^1)$\,-\,dense in $\cal L^\ii$,
$\wt T|_{\cal L^1\cap\cal L^\ii}$ uniquely extends to an operator on $\cal L^\ii$ which coincides with $(\wh{T^*})^*:\cal L^\ii\to\cal L^\ii$.

Let us now show that $\|\wt T\|_{\cal L^\ii\to\cal L^\ii}\leq 1$.  Indeed, given $f\in\cal L^1\cap\cal L^\ii_+$, we have
\[
\int_\Om f\wt T(\mb 1)d\mu=\int_\Om f(\wh{T^*})^*(\mb 1)d\mu=\int_\Om\wh{T^*}(f)d\mu=\int_\Om T^*(f)d\mu
\leq \int_\Om fd\mu,
\]
and we conclude that $\wt T(\mb 1)\leq\mb 1$, hence $\|\wt T(\mb 1)\|_\ii\leq 1$. Therefore, in view of Theorem \ref{t30} with $\cal A=\cal B=\cal L^\ii$, we have
\[
\|\wt T\|_{\cal L^\ii\to\cal L^\ii}=\|\wt T(\mb 1)\|_\ii\leq 1.
\]
This completes the proof of the theorem in the case $T\in DS^+$, since the operator $\wt T: \cal L^1+\cal L^\ii\to\cal L^1+\cal L^\ii$ defined by
\[
\wt T(f)=T(f) \ \ \forall \ f\in\cal L^1,\ \ \ \wt T(g)=(\wh{T^*})^*(g) \ \ \forall \ \ g\in\cal L^\ii
\]
satisfies the required conditions.

Let now $T\in DS$. Since $|T|\in DS^+$, it follows as above that $|T|^*:\cal L^\ii\to\cal L^\ii$ uniquely extends to a positive continuous linear operator $\wh{|T|^*}: \cal L^1\to\cal L^1$ and, since, by Proposition \ref{p31},
\[
\|T^* f\|_1 \leq  \| |T^*| (f)\|_1=\| |T|^* (f)\|_1=\| \wh{|T|^*}(f)\|_1 \, \ \ \forall \, \ f \in\cal L^1\cap\cal L^\ii_+,
\]
$T^*$ is $\|\cdot\|_1$\,-\,continuous on $\cal L^1\cap\cal L^\ii$. Therefore, $T^*$ admits a unique $\|\cdot\|_1$\,-\,continuous extension $\wh{T^*}$ to
$\cal L^1$, implying as above that $\wt T=(\wh{T^*})^*$ is the unique extension of $T$ to $\cal L^\ii$.

Next, $\wh{T^*}(f)=T^*(f)$  for all $f\in\cal L^1\cap\cal L^\ii$ implies that
\[
|\wh{T^*}|(f)=|T^*|(f)=|T|^*(f)=\wh{|T|^*}(f), \ \ f\in\cal L^1\cap\cal L^\ii,
\]
hence $|\wh{T^*}|(g)=\wh{|T|^*}(g)$ for all $g\in\cal L^\ii$, since $|T|^*$ is $\sg(\cal L^\ii,\cal L^1)$\,-\,continuous on $\cal L^1\cap\cal L^\ii$. Since, as above, we have
\[
\|\,(\wh{|T|^*})^*\,\|_{\cal L^\ii\to\cal L^\ii}\leq 1,
\]
it now follows by Proposition \ref{p31} that
\[
\begin{split}
\|\wt T\|_{\cal L^\ii\to\cal L^\ii}&=\|(\wh{T^*})^*\|_{\cal L^\ii\to\cal L^\ii}=\|\,|(\wh{T^*})^*|\,\|_{\cal L^\ii\to\cal L^\ii}\\
&=\|\,|\wh{T^*}|^*\|_{\cal L^\ii\to\cal L^\ii}=\|\,(\wh{|T|^*})^*\,\|_{\cal L^\ii\to\cal L^\ii}\leq 1,
\end{split}
\]
completing the proof.
\end{proof}

\begin{rem}
Theorem \ref{t31} implies that one can  (and we will in what follows) assume without loss of generality that any $T\in DS$ is defined on entire space $\cal L^1+\cal L^\ii$ and satisfies conditions
\begin{equation}\label{e31}
\| T(f)\|_1\leq\| f\|_1\, \ \ \forall \ \,f\in\cal L^1\text{ \ \ and \ \ } \| T(f)\|_\ii\leq\| f\|_\ii\, \ \ \forall \ \,f\in\cal L^\ii.
\end{equation}

\end{rem}

\section{Almost uniform convergence of Besicovitch weighted averages}

In this section we will show that pointwise convergence of Besicovitch weighted ergodic
averages (see, for example, \cite{clo}) can be extended to the context of a.u. convergence and a Dunford-Schwartz operator acting in $\cal R_\mu$ (Theorem 1.4 below).

Let $\Bbb C_1$ be the unit circle in the field $\mathbb C$  of complex numbers, and let $\mathbb Z$ be the set of integers. A function $P : \mathbb Z \to \mathbb C$ is said to be a {\it trigonometric polynomial} if
$P(k)=\sum\limits_{j=1}^{s} z_j\lb_j^k$, $k\in \mathbb Z$, for some $s\in \mathbb N$, $\{ z_j \}_1^s \subset \mathbb C$, and $\{ \lb_j \}_1^s \subset \mathbb C_1$.
A sequence $\{ \beta_k \} \subset \Bbb C$ is called a {\it bounded Besicovitch sequence} if

(i) $| \beta_k | \leq C<\ii$ for all $k \in \mathbb N$ and some $C>0$;

(ii) for every $\ep >0$ there exists a trigonometric polynomial $P$ such that
\begin{equation}\label{e41}
\limsup_n \frac 1n \sum_{k=0}^{n-1} | \beta_k - P(k) |<\ep .
\end{equation}

Let $E$ be a Banach space, and let $A_n: E \to \cal L^0$ be a sequence of linear maps. Given $f\in E$, the function
\[
A^*(f)=\sup_n|A_n(f)|
\]
is called the {\it maximal function} of $f$. If $A^*(f)\in\cal L^0$ for every $f\in E$, then the function
\[
A^*: E\to \cal L^0, \ \ f\in E,
\]
is called the {\it maximal operator} of the sequence $\{A_n\}$. 

Here is the well-known maximal ergodic inequality for the sequence $\{A_n(T)\}$, $T\in DS$ (see,  for example,  \cite[Theorem 3.3]{cl1}):

\begin{teo}\label{t41}
Let $T\in DS$. If
\[
A(T)^*(f)=\sup\limits_n\left|A_n(T)(f)\right|, \ \ f \in \cal L^1,
\]
the maximal operator of the sequence $\{A_n(T)\}$ on $E=\cal L^1$, then
\[
\mu\{A(T)^*(|f|)>\lb\}\leq\frac{\|f\|_1}\lb \text{ \ \ for all \ } f\in \cal L^1, \ \lb>0.
\]
\end{teo}

Given $T\in DS$, $\{\bt_k\}\su \mathbb C$, and $f\in\cal L^1+\cal L^\ii$, denote
\begin{equation}\label{e42}
B_n(f)=B_n(T)(f)=\frac1n\sum_{k=0}^{n-1}\bt_kT^k(f).
\end{equation}

\begin{cor}\label{c41}
Let $\{\bt_k\}\su\mathbb C$ be such that $|\bt_k|\leq C<\ii$ for every $k$. If $T\in DS$, then
\[
\mu\{B_n(T)^*(|f|)>\lb\}\leq 6C\frac{\|f\|_1}\lb \ \ \ \forall \ \,f\in \cal L^1, \ \lb>0.
\]
\end{cor}
\begin{proof}
We have
\[
B_n(T)=\frac1n\sum_{k=0}^{n-1}(\operatorname{Re}\bt_k+C)T^k+\frac{i}n\sum_{k=0}^{n-1}(\operatorname{Im}\bt_k+C)T^k-C(1+i)A_n(T).
\]
Therefore, as $0\leq\operatorname{Re}\bt_k+C\leq2C$ and $0\leq\operatorname{Im}\bt_k+C\leq2C$ for every $k$, it follows that
\[
|B_n(T)(f)|\leq6CA_n(|T|)(|f|)\text{ \ for every\ }f\in \cal L^1+\cal L^\ii\text{ \ and\ }n,
\]
and Theorem \ref{t41} implies that
\[
\begin{split}
\mu\{B(T)^*(|f|)>\lb\}&=\mu\left\{\sup_n|B_n(T)(|f|)|>\lb\right\}\leq\mu\left\{6C\sup_n|A_n(|T|)(|f|)|>\lb\right\}\\
&=\mu\left\{A(|T|)^*(|f|)>\frac{\lb}{6C}\right\}\leq 6C\frac{\|f\|_1}\lb.
\end{split}
\]
\end{proof}

Let us denote
\[
\cal L^0_\mu=\left\{f\in\cal L^0: \, \mu\{| f |>\lb\}<\ii\text{ \ for some\  }\lb>0\right\}.
\]

\begin{pro}[see \cite{cl1}, Proposition 3.1]\label{p41}
The $*$-subalgebra $\cal L^0_\mu$ of $\cal L^0$ is complete with respect to a.u. convergence.
\end{pro}

In what follows $t_\mu$ will stand for the {\it measure topology} in $\cal L^0$, that is, the topology given by the following system of neighborhoods of zero:
\[
\cal N(\ep,\dt)=\{ f\in\cal L^0: \ \mu\{|f|>\dt\}\leq\ep\}, \ \ \ep>0, \ \dt>0.
\]
It is well-known that $(\cal L^0,t_\mu)$ is a complete metrizable topological vector space. Since $\cal L^0_\mu$ is a closed linear subspace of $(\cal L^0,t_\mu)$, it follows that $(\cal L^0_\mu,t_\mu)$ is also a complete metrizable topological  vector space.

A proof of the next fact is given in \cite[Lemma 3.1]{cl1}.
\begin{lm}\label{l42}
Let $(E,\|\cdot\|)$ be a Banach space. If the maximal operator $A^*: E\to\cal L^0$ of a sequence of linear maps $A_n: (E,\|\cdot\|)\to (\cal L^0_\mu,t_\mu)$ is continuous at zero, then the set
\[
E_c=\{f\in E: \, \{A_n(f)\}\text{\ converges a.u.}\}
\]
is closed in $E$.
\end{lm}

Since Corollary \ref{c41} entails that the sequence $B_n(T): (\cal L^1,\|\cdot\|_1)\to(\cal L^0_\mu,t_\mu)$ is continuous at zero for every $T\in DS$, we arrive at the following.

\begin{cor}\label{c42}
If $T\in DS$ and $\{\bt_k\}\su \mathbb C$ is such that $\bt_k\leq C<\ii$ for all $k$, then the set
\[
\cal L^1_c=\left\{f\in \cal L^1:\, \{B_n(T)(f)\} \text{\ converges a.u.}\right\}
\]
is closed in $\cal L^1$.
\end{cor}

Note that Proposition \ref{p21} implies that $T(\cal R_\mu)\su\cal R_\mu$ for any $T\in DS$. The following theorem was established in \cite[Theorems 3.1, 3.4]{cl1} (see also \cite{kk}) under the initial assumption that the operator $T$ satisfied conditions (\ref{e31}). Also, even though it was proved for real-valued functions, the argument remains valid in the general case.

\begin{teo}\label{t42}
If $T\in DS$, then for every $f\in\cal R_\mu$ the averages $A_n(T)(f)$ converge a.u. to some $\wh f\in\cal R_\mu$.
Conversely, if  \ $f\in (\cal L^1+ \cal L^\ii) \setminus \cal R_\mu$, then there exists \ $T\in DS$ such that the sequence $\{A_n(T)(f)\}$ does not converge a.e., hence a.u.
\end{teo}

In particular, Theorem \ref{t42} entails that Dunford-Schwartz pointwise ergodic theorem holds for $f\in\cal L^1+\cal L^\ii$ and for any $T \in DS$ if and only if $f\in \cal R_\mu$.

\begin{lm}\label{l43}
Let $(X,\nu)$ and $Y,\mu)$ be $\sigma$-finite measure spaces, and let $\{g_n\}\su\cal L^0(X\otimes Y,\nu\otimes\mu)$ be such that $g_n\to g$ a.u. on $X\otimes Y$. Then $g_n(x,\cdot)\to g(x,\cdot)$ a.u. on $Y$ for almost all $x\in X$.
\end{lm}
\begin{proof}
Fix $\ep>0$. Given $k\in\mathbb N$, there exists $G_k\su X\otimes Y$ such that
\[
(\nu\otimes\mu)((X\otimes Y)\sm G_k)<\frac{\ep^2}k \text{ \ \ and \ \ } \|(g-g_n)\chi_{G_k}\|_\ii\to 0\text{\, \ as\,\ }n\to\ii.
\]
If $x\in X$ and
\[
G_k(x)=\{y\in Y:\, (x,y)\in G_k\},
\]
then we have
\[
\begin{split}
\frac{\ep^2}k&>(\nu\otimes\mu)((X\otimes Y)\sm G_k)=\int_X\mu(Y\sm G_k(x))d\nu(x)\\
&\ge\int_{X\sm  X_k}\mu(Y\sm G_k(x))d\nu(x).
\end{split}
\]
Therefore, it follows that
\[
\nu(X\sm X_k)<\frac{\ep}k\text{ \ \ for \ }X_k=\{x\in X:\, \mu(Y\sm G_k(x))<\ep\},
\]
implying that if $X^\prime=\cup_kX_k$, then
\[
\nu(X\sm X^\prime)=0.
\]

Now, if $x\in X^\prime$, then $x\in X_{k_0}$ for some $k_0$, so, if $Y_x=G_{k_0}(x)$, then $\mu(Y\sm Y_x)<\ep$ and
\[
\|(g(x,\cdot )-g_n(x,\cdot))\chi_{Y_x}\|_{\cal L^\ii(Y)}\leq\|(g(x,\cdot)-g_n(x,\cdot))\chi_{G_{k_0}}\|_{\cal L^\ii(X\otimes Y)}\to 0,
\]
that is, $g_n(x,\cdot)\to g(x,\cdot)$ a.u. on $Y$.
\end{proof}

The following fact can be easily verified.
\begin{lm}\label{l41}
 Let a sequence $\{ b_n\}\su\cal L^\ii$ be such that, given $\ep>0$,
there exists an a.u. convergent sequence $\{a_n \}\su\cal L^\ii$ for which the inequality
\[
\| b_n - a_n \|_\ii\leq\ep
\]
\noindent
holds for all big enough $n$. Then the sequence $\{ b_n \}$ itself converges a.u.
\end{lm}

\begin{teo}\label{t43}
Let $T\in DS$, and let $\{\beta_k\}$ be a bounded Besicovitch sequence. Then for every $f\in\cal L^1$ the averages (\ref{e42}) converge a.u.
\end{teo}

\begin{proof} In view of Corollary \ref{c42}, in order to prove that the averages $B_n(T)$ converge a.u. in $\cal L^1$ for every $T\in DS$, it is sufficient to present a dense subset $D$ of $\cal L^1$ such that the sequence $\{B_n(T)(f)\}$ converges a.u. for each $f\in D$.

Following the scheme in \cite{ry}, we begin by showing that,
given a trigonometric polynomial $P$ and $f\in \cal L^1$, the averages
\[
A_n^{(P)}(T)(f)=\frac 1n \sum_{k=0}^{n-1} P(k) T^k(f)
\]
converge a.u.  Consider the product space $(\Bbb C_1, \nu)\otimes (\Om, \mu)$, where $\nu$ is
Lebesgue measure in $\Bbb C_1$. Fix $\lb \in \Bbb C_1$ and define an operator $T_{\lb}$ on $\cal L^1(\Bbb C_1 \otimes \Om)$ as follows: if $\wt f\in \cal L^1(\Bbb C_1 \otimes \Om)$, $z\in \Bbb C_1$, and $\om \in \Om$, we put
$$
T_{\lb}(\wt f)(z,\om)=T(f_{\lb z})(\om), \text { \ where \ } f_z(\om)=\wt f(z,\om)
$$
\noindent
(note that $f_z\in \cal L^1$ for almost all $z\in \Bbb C_1$). It is easily verified that $T_{\lb}\in DS$ on $\cal L^1(\Bbb C_1 \otimes \Om)+\cal L^{\ii}(\Bbb C_1 \otimes \Om)$. For instance, given $\wt f\in \cal L^1(\Bbb C_1 \otimes \Om)$, we have
\[
\begin{split}
\int\limits_{\Bbb C_1\otimes\Om}&\big|T_\lb(\wt f)(z,\om)\big|d (\nu\otimes\mu)
=\int\limits_{\Bbb C_1}\int\limits_{\Om}\big|T(f_{\lb z})(\om)\big| d \mu\,d \nu\leq \int\limits_{\Bbb C_1}\int\limits_\Om\big|f_{\lb z}(\om)\big|  d \mu\,d \nu\\
&=\int\limits_\Om\int\limits_{\Bbb C_1}\big|f_{\lb z}(\om)\big|d\nu\,d \mu =\int\limits_\Om\int\limits_{\Bbb C_1}\big|f_z(\om)\big|d \nu\,d \mu =\int\limits_{\Bbb C_1\otimes\Om}\big|\wt f(z,\om)\big|d (\nu\otimes\mu)=\|\wt f\|_1,
\end{split}
\]
hence $T_{\lb}(\wt f)\in \cal L^1(\Bbb C_1 \otimes \Om)$ and $\| T_{\lb}(\wt f)\|_1\leq \| \wt f\|_1$.

\noindent
It follows by induction that
$$
(T_{\lb}^k(\wt f))_z=T^k(f_{\lb^kz}), \ k=1,2, \dots
$$
\noindent
Indeed, we have $(T_{\lb}(\wt f))_z(\om)=T_{\lb}(\wt f)(z,\om)=T(f_{\lb z})(\om)$, so that \\ $(T_{\lb}(\wt f))_z=T(f_{\lb z}$),
and if $(T_{\lb}^k(\wt f))_z=T^k(f_{\lb^kz})$ for some $k\in \Bbb N$, then
$$
T_{\lb}^{k+1}(\wt f)_z(\om)
=T_{\lb}(T_{\lb}^k(\wt f))(z,\om) =T(T_{\lb}^k(\wt f)_{\lb z})(\om)
$$
$$
=T(T^k(f_{\lb^{k+1}z}))(\om) =T^{k+1}(f_{\lb^{k+1}z})(\om).
$$
Therefore, one can write
$$
T_{\lb}^k(\wt f)(z,\om)=(T_{\lb}^k(\wt f))_z(\om)=T^k(f_{\lb^kz})(\om), \  k=1,2, ...
$$
\noindent
Now, if $\wt f\in \cal L^1(\Bbb C_1 \otimes \Om)$ is given by $\wt f(z,\om)=zf(\om)$,   then \\ $f_{\lb^kz}(\om)=\wt f(\lb^kz,\om)=\lb^kzf(\om)$,  and we obtain
$$
T_{\lb}^k(\wt f)(z,\om)=T^k(f_{\lb^kz})(\om)=\lb^kzT^k(f(\om)), \   k=1,2, ...
$$
\noindent
By Theorem \ref{t42}, the averages
\[
\frac 1n \sum_{k=0}^{n-1}T_{\lb}^k(\wt f)(z,\om)=z\,\frac 1n \sum_{k=0}^{n-1}\lb^kT^k(f(\om))
\]
\noindent
converge a.u. on $(z,\om)\in \Bbb C_1 \otimes \Om$. Thus, by Lemma \ref{l43}, the above averages converge a.u. on $\Om$ for some $z\in\mathbb C_1$, which implies that the averages
\[
\frac 1n \sum_{k=0}^{n-1}\lb^kT^k(f)
\]
converge a.u. Therefore, by linearity, $A_n^{(P)}(T)(f)$ converge a.u.

Now, assume that $f\in D=\cal L^1\cap \cal L^\ii$. If we fix $\ep >0$ and take $P$ to satisfy the inequality (\ref{e41}), then
\[
\| A_n^{(P)}(T)(f) - B_n(T)(f)\|_{\ii}\leq \| f \|_{\ii}\,\frac 1n \sum_{k=0}^{n-1} | \beta_k - P(k) |< \ep\,\| f \|_{\ii}
\]
for all big enough $n$. Thus, Lemma \ref{l41} entails a.u. convergence of the sequence $\{B_n(T)(f)\}$, which completes the proof since the set $D$ is dense in $\cal L^1$.
\end{proof}

Now we can present the main result of the section:

\begin{teo}\label{t44}
Let $T\in DS$, and let $\{\bt_k\}$ be a bounded Besicovitch sequence. Then, given $f\in\cal R_\mu$, the averages (\ref{e42}) converge a.u. to some $\wh f\in\cal R_\mu$.
\end{teo}
\begin{proof}
Let $C\neq 0$ be such that $\sup\,\{|\bt_k|\}\leq C$. Fix $\ep>0$ and $\dt>0$. In view of Proposition \ref{p21}, there exist $g\in\cal L^1$ and $h\in\cal L^\ii$ such that
\[
f=g+h, \ \ g\in\cal L^1, \text{\ \ and\ \ } \| h\|_\ii\leq\frac\dt{3C}.
\]
Since $g\in\cal L^1$, Theorem \ref{t43} implies that there exists $E\su \Om$ and $N\in \mathbb N$ satisfying conditions
\[
\mu(\Om\sm E)\leq\ep \text{ \ \ and \ \ } \|(B_m(g)-B_n(g))\chi_E\|_\ii\leq\frac\dt3\, \ \ \ \forall\ \ m,n\ge N.
\]
Then, given $m,n\ge N$, we have
\[
\begin{split}
\|(B_m(f)-B_n(f))\chi_E\|_\ii&\leq\|(B_m(g)-B_n(g))\chi_E\|_\ii+\|(B_m(h)-B_n(h))\chi_E\|_\ii\\
&\leq\frac\dt3+\|B_m(h)\|_\ii+\|B_n(h)\|_\ii\leq \frac\dt3+2C\|h\|_\ii\leq\dt,
\end{split}
\]
implying, by Propositions \ref{p41} and \ref{p22}, that the sequence $\{M_n(f)\}$ converges a.u. to some
$\wh f\in\cal R_\mu$.
\end{proof}

\section{Wiener-Wintner-type ergodic theorem in $\cal R_\mu$}

Recall that
$(\Om,\mu)$ is a $\sg$\,-\,finite measure space, and let $\tau:\Om\to\Om$ be a measure preserving transformation (m.p.t.). Assume that $(X,\nu)$ is a finite measure space and $\phi:X\to X$ is also a m.p.t. Given $f:\cal L^0$ and $g\in\cal L^1(X)$, denote
\begin{equation}\label{eq51}
A_n(f,g)(\om,x)=\frac 1n\sum_{k=0}^{n-1}g(\phi^kx)f(\tau^k\om).
\end{equation}
Here is an extension of Bourgain's Return Times theorem to infinite measure \cite[p.\,101]{as0}.

\begin{teo}\label{t51}
Let $F\su \Om$, $\mu(F)<\ii$. Then there exists $\Om_F\su\Om$ such that $\mu(\Om\sm\Om_F)=0$ and for any $(X,\nu,\phi)$ and $g\in\cal L^1(X)$ the averages
\[
A_n(\chi_F,g)(\om,x)=\frac1n\sum_{k=0}^{n-1}g(\phi^kx)\chi_F(\tau^k\om)
\]
converge $\nu$\,-\,a.e. for all $\om\in\Om_F$.
\end{teo}

The next theorem is  a version of Theorem \ref {t51} where the functions $\chi_F$ and $g\in\cal L^1(X)$
are replaced by $f\in\cal L^1(\Om)$ and  $g\in\cal L^\ii(X)$, respectively.

\begin{teo}\label{t52}
Given $f\in \cal L^1(\Om)$, there exists a set $\Om_f\su \Om$ with $\mu(\Om\sm \Om_f)=0$ such that for any $(X,\nu,\phi)$ and $g\in \cal L^\ii(X)$ the averages (\ref{eq51}) converge $\nu$\,-\,a.e. for all $\om\in\Om_f$.
\end{teo}
\begin{proof} Let $f\in \cal L^1(\Om)$. Then there exist $\{ \lb_{m,i}\} \su \mathbb C$ and $F_{m,i}\su \Om$ with $\mu(F_{m,i})<\ii$, $m=1,2,\dots$, $1\leq i\leq l_m$, such that
\[
\| f-f_m\|_1\to 0, \text{ \ where \ } f_m=\sum_{i=1}^{l_m}\lb_{m,i}\chi_{F_{m,i}}.
\]
If
\[
\Om_{m,j}=\left \{\om\in\Om: \ \sup_n\frac1n\sum_{k=0}^{n-1}|f-f_m|(\tau^k\om)>\frac1j\right\},
\]
then, due to the maximal ergodic inequality, we have
\[
\mu(\Om_{m,j})\leq j\|f-f_m\|_1,
\]
which implies that $\mu (\cap_m\Om_{m,j})=0$ for a fixed $j$. Therefore, denoting
\[
\Om_0=\Om\sm\left(\cup_j\cap_m \Om_{m,j}\right),
\]
we obtain $\mu(\Om\sm \Om_0)=0$.

If $\om\in\Om_0$, then $\om\notin \Om_{m_j,j}$ for every $j$ and some $m_j$ and, therefore,
\begin{equation}\label{eq52}
\sup_n\frac 1n \sum_{k=0}^{n-1}|f-f_{m_j}|(\tau^k\om)\leq \frac 1j \text{ \ for all \ } j \text{ \ and \ }\om\in\Om_0.
\end{equation}

Now, by Theorem \ref{t51}, there exist $\Om_{j,i}\su \Om$ with $\mu(\Om \sm \Om_{j,i})=0$ such that for every $(X,\nu,\phi)$ and $g\in \cal L^\ii(X)$ the averages
\[
\frac 1n\sum_{k=0}^{n-1}g(\phi^kx)\chi_{F_{m_j,i}}(\tau^k\om)
\]
converge $\nu$-a.e. for all $\om\in \Om_{j,i}$. Then, letting
\[
\Om_f=\left(\cup_{j=1}^\ii\cap_{i=1}^{l_{m_j}}\Om_{j,i}\right)\cap\Om_0,
\]
we obtain $\mu(\Om\sm\Om_f)=0$.

If we pick any $(X,\nu,\phi)$ and $g\in \cal L^\ii(X)$, then the averages $A_n(f_{m_j},g)(\om,X)$ converge $\nu$\,-\,a.e. for every $j$ and all $\om\in\Om_f$, and it follows that there are $X_0\su X$ with $\nu(X\sm X_0)=0$ and $C>0$
such that $|g(\phi^kx)|\leq C$ for all $k$ and $x\in X_0$ and
\[
\liminf_n\operatorname{Re}A_n(f_{m_j},g)(\om,x)=\limsup_n\operatorname{Re}A_n(f_{m_j},g)(\om,x),
\]
\[
\liminf_n\operatorname{Im}A_n(f_{m_j},g)(\om,x)=\limsup_n\operatorname{Im}A_n(f_{m_j},g)(\om,x)
\]
for all $x\in X_0$, $k$, and $\om\in \Om_f$.

Let $\om\in\Om_f$ and $x\in X_0$. Given $k$, taking into account (\ref{eq52}), we have
\[
\begin{split}
\Delta(\om,x)&=\limsup_n\operatorname{Re}A_n(f,g)(\om,x)-\liminf_n\operatorname{Re}A_n(f,g)(\om,x)\\
&=\limsup_n\operatorname{Re}A_n(f-f_{m_j},g)(\om,x)-\liminf_n\operatorname{Re}A_n(f-f_{m_j},g)(\om,x)\\
&\leq 2\sup_n A_n(|f-f_{m_j}|,|g|)(\om,x)\leq 2C \sup_n\frac1n\sum_{k=0}^{n-1}|f-f_{m_j}|(\tau^k\om)\leq\frac{2C}j.
\end{split}
\]
Therefore, $\Delta(\om,x)=0$. Similarly,
\[
\limsup_n\operatorname{Im}A_n(f,g)(\om,x)=\liminf_n\operatorname{Im}A_n(f,g)(\om,x),
\]
and we conclude that the averages (\ref{eq51}) converge $\nu$\,-\,a.e. for all $\om\in\Om_f$.
\end{proof}

Now we extend Theorem \ref{t52} to $\cal R_\mu$.
\begin{teo}\label{t53}
Given $f\in\cal R_\mu$, there exists a set $\Om_f\su\Om$ with $\mu(\Om\sm\Om_f)=0$ such that for any finite measure space $(Y,\nu)$, any m.p.t. $\phi: X\to X$, and any $g\in\cal L^\ii(X)$ the averages (\ref{eq51}) converge $\nu$\,-\,a.e. for all $\om\in\Om_f$.
\end{teo}
\begin{proof}
Due to Proposition \ref{p21}, given a natural $m$, there exists $f_m\in \cal L^1(\Om)$ and $h_m\in\cal L^\ii(\Om)$ such that $f=f_m+h_m$ and $\| h_m\|_\ii\leq \frac 1m$. Then there is $\Om_0 \su \Om$ such that $\mu(\Om\sm\Om_0)=0$ and $|h_m(\om)|\leq \frac 1m$ for all $m$ and $\om\in \Om_0$.

By Theorem \ref{t52}, as $\{ f_m\}_{m=1}^\ii\su\cal L^1(\Om)$, for every $m$ there is a set $\Om_m\su\Om$ with
$\mu(\Om \sm \Om_m)=0$ such that for every $(X,\nu,\phi)$ and $g\in\cal L^1(X)$ the averages
\begin{equation}\label{eq53}
A_n(f_m,g)(\om,x)=\frac 1n\sum_{k=0}^{n-1}g(\phi^kx)f_m(\tau^k\om)
\end{equation}
converge $\nu$\,-\,a.e. for all $\om\in \Om_m$. Therefore, if $\Om_f=\cap_{m=0}^\ii \Om_m$, then $\mu(\Om \sm \Om_f)=0$, $|h_m(\om)|\leq \frac 1m$ for all $m$ and $\om\in \Om_f$, and for every $(X,\nu,\phi)$ and $g\in \cal L^1(X)$, the averages (\ref{eq53}) converge $\nu$\,-\,a.e. for all $m$ and $\om\in\Om_f$.

Fix $\om\in\Om_f$, $(X,\nu,\phi)$, $g\in\cal L^1(X,\nu)$ and show that the averages (\ref{eq51}) converge $\nu$\,-\,a.e.
Indeed, as the averages (\ref{eq53}) converge $\nu$\,-\,a.e. for each $m$, there is a set $X_1\su X$ with $\nu(X\sm X_1)=0$ such that the sequence (\ref{eq53}) converges for every $m$ and $x\in X_1$. Also, since the averages
\[
\frac 1n\sum_{k=0}^{n-1}|g|(\phi^kx)
\]
converge $\nu$\,-\,a.e., there is a set $X_2\su X$ such that $\nu(X\sm X_2)=0$ and the sequence
$\frac1n\sum\limits_{k=0}^{n-1}|g|(\phi^kx)$ converges for all $x\in X_2$. Then, letting $X_0=X_1\cap X_2$, we conclude that $\nu(X\sm X_0)=0$, \ $\sup\limits_n\frac 1n\sum\limits_{k=0}^{n-1}|g|(\phi^kx)<\ii$, and the sequence (\ref{eq53}) converges for all $m$ and $x\in X_0$. Now, if $x\in X_0$, we  have
\[
\liminf_n\operatorname{Re}A_n(f_m,g)(\om,x)=\limsup_n\operatorname{Re}A_n(f_m,g)(\om,x),
\]
\[
\liminf_n\operatorname{Im}A_n(f_m,g)(\om,x)=\limsup_n\operatorname{Im}A_n(f_m,g)(\om,x),
\]
which implies that, for every $m$,
\[
\begin{split}
\Delta(\om,x)&=\limsup_n\operatorname{Re}A_n(f,g)(\om,x)-\liminf_n\operatorname{Re}A_n(f,g)(\om,x)\\
&=\limsup_n\operatorname{Re}A_n(f-f_{m},g)(\om,x)-\liminf_n\operatorname{Re}A_n(f-f_{m},g)(\om,x)\\
&\leq  2\sup_n \frac1n\sum_{k=0}^{n-1}|g(\phi^k x)|\cdot|h_m(\tau^k\om)|
\leq\frac2m\sup_n \frac 1n\sum_{k=0}^{n-1}|g|(\phi^k x).
\end{split}
\]
Therefore, $\Delta(\om,x)=0$. Similarly, 
\[
\limsup_n\operatorname{Im}A_n(f,g)(\om,x)=\liminf_n\operatorname{Im}A_n(f,g)(\om,x),
\]
and we conclude that the averages (\ref{eq51}) converge $\nu$\,-\,a.e.
\end{proof}

Letting in Theorem \ref{t53} \ $X=\mathbb C_1=\{x\in \mathbb C: |x|=1\}$ with Lebesgue measure
$\nu$, $\phi_\lb(x)=\lb x$, $x\in X$, for a given $\lb\in X$, and $g(x)=x$ whenever $x\in X$, we obtain Wiener-Wintner theorem for $\cal R_\mu$:

\begin{teo}\label{t54} 
If $f\in\cal R_\mu$, then there is a set $\Om_f\su \Om$ with $\mu(\Om\sm\Om_f)=0$ such that the averages
\[
\frac 1n \sum_{k=0}^{n-1}\lb^kf(\tau^k\om)
\]
converge for all $\om\in\Om_f$ and $\lb\in\mathbb C_1$.
\end{teo}

Let $P(k)=\sum\limits_{j=1}^s z_j\lb_j^k, \ k=0,1,2,\dots$ be a trigonometric polynomial (see Section 4). Then, by linearity, Theorem \ref{t54} implies the following.

\begin{cor}\label{c51}
Given $f\in\cal R_\mu$, there exists a set $\Om_f\su \Om$ with $\mu(\Om\sm\Om_f)=0$ such that the averages
\[
A_n(\{P(k)\},f)(\om)=\frac 1n\sum_{k=0}^{n-1}P(k)f(\tau^k\om)
\]
converge for every $\om\in\Om_f$ and any trigonometric polynomial $P(k)$.
\end{cor}

We will need the following.

\begin{pro}\label{p51}
If $f\in \cal L^1\cap \cal L^\ii$, then there exists $\Om_f\su \Om$ with $\mu(\Om\sm\Om_f)=0$ such that the averages
\begin{equation}\label{eq61}
A_n(\ol\bt,f)(\om)=\frac1n \sum_{k=0}^{n-1}\bt_kf(\tau^k\om)
\end{equation}
converge for every $\om\in \Om_f$ and any bounded Besicovitch sequence $\ol\bt=\{\bt_k\}$.
\end{pro}
\begin{proof}
By Corollary \ref{c51}, there exists a set $\Om_{f,1}\su \Om$, $\mu(\Om\sm\Om_{f,1})=0$, such that the sequence
$\frac 1n\sum\limits_{k=0}^{n-1}P(k)f(\tau^k\om)$ converges for every $\om\in \Om_{f,1}$ and any trigonometric polynomial $P(k)$. Also, since $f\in \cal L^\ii$, there is a set $\Om_{f,2}\su \Om$, $\mu(\Om\sm \Om_{f,2})=0$,
such that $|f(\tau^k\om)|\leq \| f\|_\ii$ for every $k$ and $\om\in \Om_{f,2}$. If we set $\Om_f=\Om_{f,1}\cap \Om_{f,2}$, then $\mu(\Om\sm\Om_f)=0$.

Now, let $\om\in\Om_f$, and let $\ol\bt=\{\bt_k\}$ be a Besicovitch sequence. Fix $\ep>0$, and choose a trigonometric polynomial
$P(k)$ to satisfy condition (\ref{e41}). Then we have
\[
\begin{split}
\Delta(\om)&=\limsup_n\operatorname{Re}A_n(\ol\bt,f)(\om)-\liminf_n\operatorname{Re}A_n(\ol\bt,f)(\om)\\
&=\limsup_n\operatorname{Re}A_n(\{\bt_k-P(k)\},f)(\om)-\liminf_n\operatorname{Re}A_n(\{\bt_k-P(k)\},f)(\om)\\
&\leq 2\| f\|_\ii\sup_n\frac 1n \sum_{k=0}^{n-1}|\bt_k-P(k)|<2\| f\|_\ii\ep
\end{split}
\]
for all sufficiently large $n$. Therefore, $\Delta(\om)=0$, and we conclude that the sequence
$\left\{\operatorname{Re}A_n(\ol\bt,f)(\om)\right\}$ converges. Similarly, we obtain convergence of the sequence $\left\{\operatorname{Im}A_n(\ol\bt,f)(\om)\right\}$, which completes the proof.
\end{proof}

\begin{teo}\label{t56}
If $f\in\cal L^1$, then there exists a set $\Om_f\su \Om$ with $\mu(\Om\sm \Om_f)=0$, such that the averages
(\ref{eq61}) converge for every $\om\in\Om_f$ and any bounded Besicovitch sequence $\ol\bt=\{ \bt_k\}$.
\end{teo}
\begin{proof}
Let a sequence $\{f_m\}\su\cal L^1\cap\cal L^\ii$ be such that $\|f-f_m\|_1\to 0$. As in the proof of Theorem \ref{t52}, we construct a subsequence $\{f_{m_j}\}$ and a set $\Om_0\su\Om$ with $\mu(\Om\sm\Om_0)=0$ such that
\[
\sup_n\frac1n\sum_{k=0}^{n-1}|f-f_{m_j}|(\tau^k\om)\leq \frac 1j\ \ \ \forall\, \ j \text{ \ and \ }\om\in\Om_0.
\]
By Proposition \ref{p51}, given $j$, there is $\Om_j\su \Om$ with $\mu(\Om\sm \Om_j)=0$ such that the sequence
$\left\{\frac 1n\sum\limits_{k=0}^{n-1}\bt_kf_{m_j}(\tau^k\om)\right\}$ converges for every $\om\in \Om_j$ and any Besicovitch sequence $\{\bt_k\}$.

If we set $\Om_f=\cap_{j=1}^\ii \Om_j \cap\Om_0$, then $\mu(\Om\sm \Om_f)=0$, and for any $\om\in \Om_f$ and any bounded Besicovitch sequence $\{ \bt_k\}$ such that $\sup_k|\bt_k|\leq C$ we have
\[
\begin{split}
\Delta(\om)&=\limsup_n\operatorname{Re}A_n(\ol\bt,f)(\om)-\liminf_n\operatorname{Re}A_n(\ol\bt,f)(\om)\\
&=\limsup_n\operatorname{Re}A_n(\ol\bt,f-f_{m_j})(\om)-\liminf_n\operatorname{Re}A_n(\ol\bt,f-f_{m_j})
(\om)\\
&\leq 2\sup_n\frac 1n \sum_{k=0}^{n-1}|\bt_k| \, |f-f_{m_j}|(\tau^k\om)\leq\frac{2C}j.
\end{split}
\]
Therefore, $\Delta(\om)=0$, hence the sequence $\left\{\operatorname{Re}A_n(\ol\bt,f)(\om)\right\}$ is convergent. Similarly, we derive convergence of the sequence $\left\{\operatorname{Im}A_n(\ol\bt,f)(\om)\right\}$, and the proof is complete.
\end{proof}

Taking into account that the sequence $\{\bt_k\}$
is bounded, we obtain, as in the proof of Theorem \ref{t53}, the following extension of Wiener-Wintner theorem.

\begin{teo}\label{t57}
Given $f\in\cal R_\mu$, there exists a set $\Om_f\su \Om$ with $\mu(\Om\sm \Om _f)=0$ such that the averages (\ref{eq61}) converge for every $\om\in\Om_f$ and every bounded Besicovitch sequence $\{\bt_k\}$.
\end{teo}

\section{Applications to fully symmetric spaces}

For any $f\in\cal L^0_\mu$ the {\it non-increasing rearrangement} of $f$ is defined as
\[
f^*(t)=\inf\left\{\lambda>0: \ \mu\{|f|>\lb\right\}\leq t\}, \ \  t>0,
\]
(see \cite[Ch.\,II, \S\,2]{bs}).

Let $\nu$ be the Lebesgue measure on $(0,\ii)$. A non-zero linear subspace  $E \su \cal L^0_\nu$ with a Banach norm $\|\cdot\|_E$ is called {\it symmetric (fully symmetric)} on  $((0,\ii),\nu)$ if
\[
f \in E, \ g\in \cal L^0_\nu, \ g^*(t)\leq f^*(t) \, \ \ \forall \ \  t>0
\]
(respectively,
\[
f \in E, \ g \in \cal L^0_\nu, \ \int_0^s g^*(t) dt\leq\int_0^s f^*(t) dt \, \ \ \forall \ \ s>0
\ \ (\text{writing\ \ } g \prec\prec f)
\]
implies that $g \in E$ and $\| g\|_E\leq \| f\|_E$.

Let $(E,\| \cdot \|_E)$ be a symmetric (fully symmetric) space on $((0,\ii),\nu)$.
Define
\[
E(\Om)=E(\Om,\mu)=\left\{ f\in \cal L^0_\mu:  f^*(t) \in E\right\}
\]
and set
\[
\| f\|_{E(\Om)}=\|  f^*(t)\|_E,  \ f\in E(\Om).
\]
It is shown in \cite{ks} (see also \cite[Ch.\,3, Sec.\,3.5]{lsz}) that $(E(\Om), \|\cdot\|_{E(\Om)})$ is a Banach space and conditions $f \in E(\Om)$, $g\in L^0_\mu$, $g^*(t)\leq f^*(t)$ for every $t>0$ ($g\prec\prec f$) imply that $g\in E(\Om)$ and $\| g\|_{E(\Om)}\leq \| f\|_{E(\Om)}$. In such a case, we say that $(E(\Om), \|\cdot\|_{E(\Om)})$ is the symmetric (respectively, fully symmetric) space on  $(\Om,\mu)$ {\it generated} by the symmetric (respectively, fully symmetric) space $(E, \|\cdot\|_E)$. Throughout, if it does not cause confusion, we will write $(E, \|\cdot\|_E)$ or simply $E$ instead of $(E(\Om), \| \cdot \|_{E(\Om})$.

Immediate examples of fully symmetric spaces are the spaces $\cal L^p(\Om,\mu)$, $1\leq p\leq \ii$, with standard norms $\| \cdot \|_p$, the space $\cal L^1\cap \cal L^\ii$  with the norm
\[
\|f\|_{\cal L^1\cap \cal L^\ii}=\max\left \{ \|f\|_1,\|f\|_\ii\right\},
\]
and the space $\cal L^1+ \cal L^\ii$ with the norm
\[
\|f\|_{\cal L^1+ \cal L^\ii}=\inf \left \{ \|g\|_1+ \|h\|_\ii: \ f = g + h, \ g\in \cal L^1, \ h \in \cal L^\ii\right\}.
\]

Note that, alternatively,
\[
\cal R_\mu=\left\{f \in \cal L^1+ \cal L^\ii: \ f^*(t)\to 0 \text{ \ as \ }t\to\ii\right\}
\]
and $(\cal R_\mu, \|\cdot\|_{\cal L^1 + \cal L^\ii})$ is a symmetric space \cite[Ch.\,II, \S\,4, Lemma 4.4]{kps}.  In addition,
$\cal R_\mu$ is the closure of $\cal L^1\cap \cal  L^{\ii}$  in $(\cal L^1+\cal L^\ii, \| \cdot \|_{\cal L^1 + \cal L^\ii})$ (see \cite[Ch.\,II, \S\,3, Sec.\,1]{kps}).
Furthermore, it follows from definitions of $\cal R_\mu$ and $\| \cdot \|_{\cal L^1 + \cal L^\ii}$ that if
\[
f \in \cal R_\mu, \ \ g \in \cal L^1+\cal L^\ii\text{ \ and \ } g \prec\prec f,
\]
then $g\in \cal R_\mu$ and $\| g\|_{\cal L^1 +\cal  L^\ii} \leq \| f\|_{\cal L^1 +\cal  L^\ii}$. Therefore, $(\cal R_\mu, \|\cdot\|_{\cal L^1 +\cal  L^\ii})$ is also a fully symmetric space. If $\mu(\Om)<\ii$, then $\cal R_\mu=\cal L^1$.

Also, given $T\in DS$, we have $T(E)\su E$ and $\| T\|_{E \to E} \leq 1$ for any symmetric space $E$ (see \cite[Ch.\,II, \S\,4, Theorem\,4.1]{kps}). In addition,
\[
\int_0^s T(f)^*(t) dt \le \int_0^s f^*(t) dt\, \ \ \forall\, \ s>0,
\]
that is, $T(f)\prec\prec f$ for every $f\in\cal L^1+\cal L^\ii$ (see, for example, \cite[Ch.\,II, \S\,3, Section\,4]{kps}).

\begin{pro}\label{p61}
If $\mu(\Om)=\ii$, then a symmetric space $E $ is contained in
$\cal R_\mu$ if and only if $\mb 1\notin E$.
\end{pro}
\begin{proof}
As  $\mu(\Om) = \ii$, we have $\mb 1^*(t) = 1$ for all $t > 0$, hence $\mb 1 \notin\cal R_\mu$. Therefore,
$E$ is not contained in $\cal R_\mu$ whenever $\mb 1\in E$.

Let $\mb 1\notin  E$. If $f \in E$ and $\lim\limits_{t\rightarrow \ii} f^*(t)=\al>0$, then
\[
\mb 1^*(t) \equiv 1 \leq\al^{-1}f^*(t),
\]
implying $\mb 1\in  E$, a contradiction. Thus $\mb 1\notin E$ entails $E\su\cal R_\mu$.
\end{proof}

The following is a version of Theorems \ref{t44} for fully symmetric spaces.
\begin{teo}\label{t61}
Let $E$ be a fully symmetric space such that $\mb 1\notin E$. If $\{ \beta_k \}$ is a bounded Besicovitch sequence, then for every $T\in DS$ and $f\in E$ the averages (\ref{e42}) converge a.u. to some $\wh f\in E$.
\end{teo}
\begin{proof}
Since, by Proposition \ref{p21}, $E\su\cal R_\mu$, it follows from Theorem \ref{t44} that averages $B_n(T)(f)$ converge a.u., hence in measure topology, to some $\wh f\in\cal R_\mu$. Therefore, we have
\[
(B_n(T)(f))^*\to (\wh f)^* \ \ \text{a.e. on} \ \ (0,\ii);
\]
see, for example, \cite[Ch.\,II, \S\,2, Property $11^\circ$]{kps}.

With $M=\max\left\{1,\,\sup| \bt_k|\right\}$, we have $M^{-1}B_n(T)\in DS$, hence
\[
M^{-1}B_n(T)(f)\prec\prec  f
\]
for every $n$ \cite[Ch.\,II, \S\,3, Section 4]{kps}. Since
\[
\left(M^{-1}B_n(T)(f)\right)^*\to (M^{-1}\wh f)^*\text{ \ \ a.e.\ \ on \ }(0,s),
\]
Fatou's Lemma entails
\[
\int_0^s (M^{-1}\wh f)^*(t) dt\leq \liminf_n
\int_0^s (M^{-1}B_n(T)(f))^* dt \leq \int_0^s f^* dt
\]
for all $s>0$, that is, $(\wh f)^*\prec\prec M f^*$. As $E$ is a fully symmetric space and $f\in E$, it follows that $\wh f\in E$.
\end{proof}

The next variant of Theorems \ref{t57} for fully symmetric spaces is straightforward.
\begin{teo}\label{t62}
Let $E$ be a fully symmetric and let  $\mb 1\notin E$. Then for every $f\in E$ there exists a set $\Om_f\su \Om$ with $\mu(\Om\sm \Om _f)=0$ such that the averages (\ref{eq61}) converge for every $\om\in\Om_f$ and every bounded Besicovitch sequence $\{\bt_k\}$.
\end{teo}

A symmetric space $(E, \| \cdot \|_E)$ is said to have an {\it order continuous norm} if $\| f_n\|_E\downarrow 0$ whenever $f_n\in E$ and $f_n\downarrow 0$. It is known that a symmetric space $E$ with order continuous norm is fully symmetric and $E\su\cal R_\mu$ \cite[Ch.\,II, \S\,4]{kps}.
\begin{rem}
Since $E \subset \cal R_\mu$ for symmetric space $E$ with order continuous norm, it follows that Theorems \ref{t61} and  \ref{t62} are valid for any symmetric space  with order continuous norm.
\end{rem}

Now we give applications of Theorems \ref{t61}  and  \ref{t62}  to Orlicz, Lorentz, and Marcinkiewicz spaces.

1. Let $\Phi$ be an {\it Orlicz function}, that is, $\Phi:[0,\ii)\to [0,\ii)$ is convex, continuous at $0$ and such that $\Phi(0)=0$ and $\Phi(u)>0$ if $u\ne 0$.  Let
\[
\cal L^\Phi=\left \{ f \in\cal L^0_\mu: \  \int_\Om\Phi\left(a^{-1}|f|\right)d \mu
<\ii \text { \ for some \ } a>0 \right\}
\]
be the corresponding {\it Orlicz space}, and let
\[
\| f\|_\Phi=\inf \left \{ a>0: \ \int_{\Om}\Phi\left(a^{-1}|f|\right) \ d \mu \leq 1\right\}
\]
be the {\it Luxemburg norm} in $\cal L^\Phi$. Then $(\cal L^\Phi, \| \cdot\|_\Phi)$ is a fully symmetric space (see, for example, \cite[Ch.\,2]{es}). Since  $\mu(\Om)=\ii$, we have $\int_\Om\Phi\left (a^{-1}\right )d\mu=\ii$ for all $a>0$, hence $\mathbf 1 \notin \cal L^\Phi$.

Therefore, Theorems \ref{t61}  and  \ref{t62} hold for any Orlicz space $\cal L^\Phi$.

2.  Let $\varphi$ be an increasing  concave function on $[0,\ii)$ with $\varphi(0) = 0$ and $\varphi(t)>0$ for some $t > 0$, and let
\[
\Lb_\varphi=\left \{f \in \cal  L^0_\mu: \ \|f \|_{\Lb_\varphi} =
\int_0^\ii\mu_t(f) \ d \varphi(t)<\ii\right\},
\]
be the corresponding {\it Lorentz space}. Then $(\Lb_\varphi,\|\cdot\|_{\Lb_\varphi})$ is a fully symmetric space. In addition, if $\varphi(\ii)=\lim\limits_{t \to \infty} \varphi(t)= \ii$, then $\mathbf 1 \notin \Lb_\varphi$ (see, for example, \cite[Ch.\,II, \S\,5]{kps}).

Therefore, Theorems  \ref{t61} and \ref{t62} are valid for any Lorentz space $\Lb_\varphi$ such that $\varphi(\ii)=\ii$.

3. Let $\varphi$ be as above, and let
\[
M_\varphi=\left \{f \in \cal L^0_\mu: \ \|f \|_{M_\varphi} =
\sup\limits_{0<s< \ii}\frac{1}{\varphi(s)} \int_0^{s} f^*(t) \ d t < \infty \right \}
\]
be the corresponding {\it Marcinkiewicz space}. It is known that $(M_\varphi, \|\cdot\|_{M_\varphi})$ is a fully symmetric space, and  $\mathbf 1 \notin \Lb_\varphi$ if and only if $\lim\limits_{t \to \ii}\frac{\varphi(t)}{t}=0$ (see, for example, \cite[Ch.\,II, \S\,5]{kps}).

Thus, Theorems  \ref{t61}  and  \ref{t62} hold for any Marcinkiewicz space $M_\varphi$ such that $\lim\limits_{t \to \ii}\frac{\varphi(t)}t=0$.

\end{document}